\documentclass[12pt,psfig,reqno]{amsart}
\usepackage{amsfonts}
\usepackage{amscd}
\usepackage{amssymb,amsfonts,latexsym}
\usepackage{graphics,verbatim}
\usepackage{graphicx}
\usepackage[usenames]{color} %\color{Red}
\makeatletter
\renewcommand\section{\@startsection {section}{1}{\z@}%
                                   {-3.5ex \@plus -1ex \@minus -.2ex}%
                                   {2.3ex \@plus.2ex}%
                                   {\centering\normalfont\bf}}
 \numberwithin{equation}{section}
\numberwithin{equation}{section}
\numberwithin{equation}{section}

\theoremstyle{plain}
 \newtheorem{theorem}{Theorem}[section]
 
 \newtheorem{Prop}[theorem]{Proposition}
 \newtheorem{Lem}[theorem]{Lemma}
 \newtheorem{Cor}[theorem]{Corollary}
 
 \newtheorem{Example}[theorem]{Example}

\begin{document}

\title  {Induced measures of simple random walks on Sierpinski graphs}
 \author{Ting-Kam Leonard Wong}
\date{}

\address{Department of Mathematics\\ The Chinese University of Hong Kong\\
Shatin, NT \\ Hong Kong.}
\email{tkwong@math.cuhk.edu.hk}

\thanks {}

% \thanks{Supported by CNSF 19901025.}
 %End Topmatter

\date{\today}
\maketitle

\begin{abstract}
In \cite{[K]}, Kaimanovich defined an augmented rooted tree $(X, E)$ corresponding to the Sierpinski gasket $K$, and showed that the Martin boundary of the simple random walk $\{Z_n\}$ on it is homeomorphic to $K$. It is of interest to determine the hitting distributions $v_{{\bf x}}(\cdot) = {\mathbb{P}}_{{\bf x}}\{\lim_{n \rightarrow \infty} Z_n \in \cdot\}$ induced on $K$. Using a reflection principle based on the symmetries of $K$, we show that if the walk starts at the root of $(X, E)$, the hitting distribution is exactly the normalized Hausdorff measure $\mu$ on $K$. In particular, each $v_{{\bf x}}$, ${\bf x} \in X$, is absolutely continuous with respect to $\mu$. This answers a question of Kaimanovich [K, Problem 4.14]. The argument can be generalized to other symmetric self-similar sets.
\end{abstract}

\begin{section}{Introduction}
In \cite{[DS1]}, Denker and Sato constructed a transient Markov chain and showed that its Martin boundary can be identified as the Sierpinski gasket. This opened a new direction to study analysis on fractals via the boundary theory of random walk (\cite{[DS2]}, \cite{[DS3]}, \cite{[JLW]}, \cite{[K]}, \cite{[Ki]}, \cite{[LN]}, \cite{[LW1]}, \cite{[LW2]}, \cite{[P]}, \cite{[WL]}). The key idea is to realize a fractal as the boundary of a random walk or graph; the potential theory associated with the discrete system is used to induce on the fractal objects such as harmonic measures, harmonic functions, and Dirichlet forms.

\medskip

There are several ways to construct the random walks, and they are all related to the symbolic space of the underlying iterated function system. In \cite{[DS1]}, the transition function is one-way and gives rise to a reducible chain. This allows the Green function and Martin kernels be estimated explicitly. This construction was generalized to a class of p.c.f. fractals in \cite{[JLW]} and in a more general setting in \cite{[LW2]}. Also see \cite{[LN]} for an interesting way of assigning probabilities where the minimal Martin boundary is a proper subset of the Martin boundary.

\medskip

Another approach was provided by Kaimanovich \cite{[K]}, who introduced an {\it augmented rooted tree} (which he also called the {\it Sierpinski graph}) to realize the Sierpinski gasket as a {\it hyperbolic boundary} in the sense of Gromov \cite{[G]}. Moreover, he showed that the simple random walk on the graph has a Martin boundary homeomorphic to the Sierpinski gasket. The augmented rooted tree was generalized in \cite{[LW1]} to cover all iterated function systems satisfying the open set condition (OSC), and it was shown in \cite{[WL]} that for {\it strictly reversible} random walks the Martin boundary coincides with the hyperbolic boundary, which is the self-similar set.

\begin{figure}[h]
\begin{center}
\includegraphics[width=5.5cm,height=6cm]{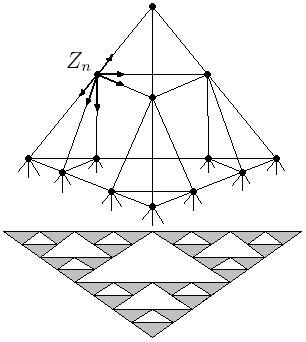}
\caption{The simple random walk $\{Z_n\}$ on $(X, E)$ induces a family of hitting distributions on the Sierpinski gasket.}
\end{center}
\end{figure}

The random walk induces naturally a family of {\it hitting distributions}, or {\it harmonic measures}, on the self-similar set. The purpose of this paper is to study these induced measures for the simple random walk on the Sierpinski graphs in \cite{[K]}. The main problem is to determine whether the measures are absolutely continuous with respect to the Hausdorff or a self-similar measure. Apart from its probabilistic interest, this question is important as the hitting distribution from the root serves as the {\it reference measure} of the {\it induced Dirichlet form} constructed in \cite{[WL]} (also see \cite{[Ki]} for the case of Cantor sets).

\medskip

Our main result answers a question of Kaimanovich [K, Problem 4.14].

\begin{theorem} \label{thm 1.1}
Consider the simple random walk $\{Z_n\}$ on the augmented rooted tree $(X, E)$ corresponding to the Sierpinski gasket $K$ in ${\mathbb{R}}^d$, $d \geq 1$. Let $Z_{\infty} = \lim_n Z_n$ be the limit of the walk on $K$ and let $\nu_{\vartheta}(\cdot) = {\mathbb{P}}_{\vartheta}\{Z_{\infty} \in \cdot\}$ be the hitting distribution where the starting point is the root $\vartheta$. Then $\nu_{\vartheta}$ equals the normalized $\alpha$-dimensional Hausdorff measure $\mu$ on $K$, where $\alpha = \dim_H K$.
\end{theorem}

Since $\{Z_n\}$ is irreducible, we immediately get the following (see [Wo, p. 221]).

\begin{Cor} \label{thm 1.2}
For all ${\bf x} \in X$, the hitting distribution $\nu_{{\bf x}}(\cdot) = {\mathbb{P}}_{{\bf x}}\{Z_{\infty} \in \cdot\}$ is absolutely continuous with
respect to $\mu$.
\end{Cor}

The main obstacle is that $Z_n$ can travel not only up and down but also through the horizontal edges. This makes it very difficult to estimate the transition probabilities and the Martin kernels. Instead of estimating them, we shall exploit the symmetries of the Sierpinski gasket and prove directly that $\nu_{\vartheta}$ satisfies two identities which we call the {\it group invariance identity} (Theorem \ref{thm 4.4}) and the {\it self-similar identity} (Theorem \ref{thm 6.5}). They force $\nu_{\vartheta}$ to be exactly $\mu$. The argument involves interesting probabilistic and algebraic constructions. Our method can be generalized to some other symmetric self-similar sets, but for simplicity we restrict ourselves to the case of Sierpinski gaskets.

\medskip

This paper is organized as follows. In Section 2, we review the construction of the Sierpinski graph $(X, E)$ and identify the Sierpinski gasket as the Martin boundary of the simple random walk $\{Z_n\}$ on $(X, E)$. For motivation, we also give a heuristic argument for the simplest case of dimension $1$, where the boundary is simply a unit interval. To generalize the argument to higher dimensions, in Section 3 we analyze the symmetries of the Sierpinski gasket and their actions on $(X, E)$. They are used to formulate the group invariance identity. In Section 4 we introduce a {\it reflection principle} and construct a coupling of $\{Z_n\}$ with a simple random walk $\{\tilde{Z}_k\}$ on the subgraph $0X$ which is isomorphic to $X$. This is used in Section 5 to prove the self-similar identity and the main theorem. Section 6 contains further remarks, including extension to other symmetric self-similar sets, and some open questions.

\end{section}

\bigskip

\begin{section}{Preliminaries and motivations}
In this section we introduce the notations and identify the Sierpinski gasket as a Martin boundary. The reader can refer to \cite{[WL]} for a detailed treatment under a more general setting. Fix $d \geq 1$ and a collection of points $\{p_i\}_{i = 0}^d$ that generates a regular simplex in ${\mathbb{R}}^d$, i.e., $|p_i - p_j| = \delta_{ij}$. The $d$-{\it dimensional} {\it Sierpinski gasket}, denoted by $K = K^d$, is the self-similar set of the iterated function system (IFS) $F_i(x) = \frac{1}{2}p_i + \frac{1}{2}x$, $i = 0, ..., d$.

\medskip

Following \cite{[K]} and \cite{[LW1]}, we define an {\it augmented rooted tree} (or {\it Sierpinski graph}) $(X, E)$ as follows. Let
\[
X = \{\vartheta\} \cup \bigcup_{n = 1}^{\infty} \{0, 1, ..., d\}^n
\]
be the symbolic space corresponding to the IFS $\{F_i\}_{i = 0}^d$. Here $\vartheta$ is the {\it empty word} which will be regarded as the root of $(X, E)$. For ${\bf x} = i_1...i_n \in X$, we let $|{\bf x}| = n$ be the {\it length} of ${\bf x}$ and define $F_{{\bf x}} = F_{i_1} \circ \cdots \circ S_{i_n}$ (we define $S_{\vartheta} = \mathrm{id}$ by convention). We let $K_{{\bf x}} = F_{{\bf x}}(K)$ be the {\it cell} corresponding to ${\bf x} \in X$. If ${\bf x} = i_1...i_n \in X \setminus \{\vartheta\}$, we use ${\bf x}^- = i_1...i_{n-1}$ to denote the {\it ancestor} of ${\bf x}$.

\medskip

The edge set $E = E_v \cup E_h$ of the Sierpinski graph consists of {\it vertical edges} $(E_v)$ and {\it horizontal edges} $(E_h)$. We define, for ${\bf x} \neq {\bf y}$,
\[
\begin{array}{lllll}
({\bf x}, {\bf y}) \in E_v & \Leftrightarrow & {\bf y} = {\bf x}^- \text{ or } {\bf x} = {\bf y}^-,\\
({\bf x}, {\bf y}) \in E_h &\Leftrightarrow& |{\bf x}| = |{\bf y}| \text{ and } F_{{\bf x}}(K) \cap F_{{\bf y}}(K) \neq \emptyset.
\end{array}
\]
We write ${\bf x} \sim {\bf y}$ if $({\bf x}, {\bf y}) \in E$. By a {\it path} in $(X, E)$, we mean a finite sequence $\{{\bf x}_n\}_{n = 0}^N$ such that ${\bf x}_n \sim {\bf x}_{n+1}$ for all $n$.

\begin{figure}[h] \label{fig 2}
\begin{center}
\includegraphics[width=11cm,height=4cm]{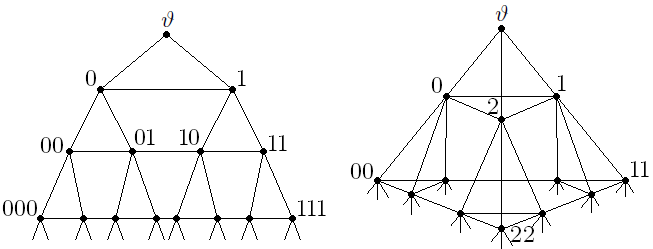}
\caption{The Sierpinski graphs corresponding to $d = 1$ (left) and $d = 2$ (right).}
\end{center}
\end{figure}

\medskip

For ${\bf x} = i_1i_2...i_n \neq \vartheta$, let $p_{{\bf x}} = F_{i_1i_2...i_{n-1}}(p_{i_n})$ be the `\textit{dyadic point}' on $K$ corresponding to ${\bf x}$. In particular, for $i \neq j$,
\[
p_{ij} = p_{ji} = F_i(p_j) = F_j(p_i) = \frac{1}{2}p_i + \frac{1}{2}p_j
\]
is the midpoint of the line segment $[p_i, p_j]$.

\medskip

Let $\{Z_n\}$ be the simple random walk on $(X, E)$ with transition function $P$, i.e.,
\[
P({\bf x}, {\bf y}) = \left\{
                        \begin{array}{ll}
                          \frac{1}{{\mathrm{deg}}({\bf x})}, & \hbox{if ${\bf x} \sim {\bf y}$,} \\
                          0, & \hbox{otherwise.}
                        \end{array}
                      \right.
\]
Here ${\mathrm{deg}}({\bf x})$, the {\it degree} of ${\bf x}$, is the number of edges connecting ${\bf x}$. The {\it Green function} of the walk is defined by
\[
G({\bf x}, {\bf y}) = \sum_{n = 0}^{\infty} P^n({\bf x}, {\bf y}), \ \ \ {\bf x}, {\bf y} \in X.
\]
Here $P^n$ is the $n$-step transition function. The {\it Martin kernel} is defined by
\[
K({\bf x}, {\bf y}) = \frac{G({\bf x}, {\bf y})}{G(\vartheta, {\bf y})}, \ \ \ {\bf x}, {\bf y} \in X.
\]
The {\it Martin compactification} of the random walk is defined as the minimal compactification $\widehat{X}$ of $X$ such that for each ${\bf x} \in X$, the function $K({\bf x}, \cdot)$ extends continuously up to the {\it Martin boundary} ${\mathcal{M}} = \widehat{X} \setminus X$ (see \cite{[Wo]}). It can be shown that under the Martin topology, $Z_n$ converges almost surely to a point $Z_{\infty}$ on the Martin boundary ${\mathcal{M}}$.

\medskip

The following result first appeared as [K, Theorem 4.7] and is a special case of [WL, Theorem 3.5].

\begin{theorem} \label{thm 2.1}
The Martin boundary of $\{Z_n\}$ is homeomorphic to the Sierpinski gasket $K$ under a canonical homeomorphism.
\end{theorem}

Hence, we may identify ${\mathcal{M}}$ with $K$ and it makes sense to talk about the limit $Z_{\infty}$ on $K$. Following \cite{[WL]}, we may express the limit in terms of a {\it projection} $\iota: X \rightarrow K$. For any ${\bf x} \in X$, pick an arbitrary point $\iota({\bf x}) \in K_{{\bf x}}$. Then under the usual topology of ${\mathbb{R}}^d$, we have
\begin{eqnarray} \label{eqn 2.2}
Z_{\infty} = \lim_{n \rightarrow \infty} \iota(Z_n).
\end{eqnarray}
For example, when $d = 1$ and $[p_0, p_1] = [0, 1]$, the sequence
\[
\vartheta, 0, 00, 001, 0011, 00111, 001111, ...
\]
converges to $\frac{1}{4}$.

\bigskip

The constructions in this paper involves some technicalities. To illustrate the main ideas, we give a heuristic argument for the case $d = 1$, where $K^1 = [p_0, p_1]$ is simply a unit interval, say $[0, 1]$. First, from the symmetry of $(X, E)$ and $[0, 1]$, we see that $\nu_{\vartheta}$ must be symmetric about $\frac{1}{2}$:
\begin{eqnarray} \label{eqn 3.1}
\nu_{\vartheta}(B) = \nu_{\vartheta}(1 - B) = \nu_{\vartheta}(RB), \ \ \ B \subset [0, 1] \  \text{ Borel}.
\end{eqnarray}
Here $Rx = 1 - x$ is the reflection about $\frac{1}{2}$ and is a symmetry of $[0, 1]$. We call (\ref{eqn 3.1}) a {\it group invariance identity}.

\medskip

Next, we observe that the subgraph $0X = \{0{\bf x}: {\bf x} \in X\}$ is isomorphic to $X$. We will use $\{Z_n\}$, starting at $\vartheta$, to construct a simple random walk $\{\tilde{Z}_k\}$ on $0X$ starting at $0$. The reflection $R$ induces naturally a reflection, also denoted by $R$, on $X$ which flips all the symbols. For example, $R\vartheta = \vartheta$ and $R(00101) = 11010$. Consider the {\it reflected random walk} $\{Z^R_n\}$ defined by
\[
Z^R_n = \left\{
                \begin{array}{ll}
                  Z_n, & \hbox{if $Z_n \in \{\vartheta\} \cup 0X$} \\
                  RZ_n, & \hbox{if $Z_n \in 1X$}
                \end{array}
              \right., \ \ \ n \geq 0.
\]

\begin{figure}[h]
\begin{center}
\includegraphics[width=11.5cm,height=5.5cm]{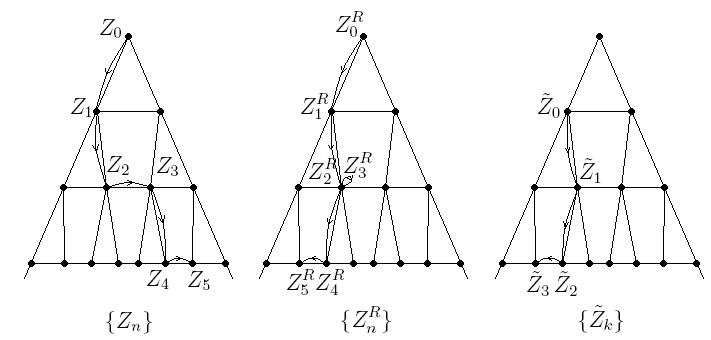}
\caption{The transformation $\{Z_n\} \mapsto \{\tilde{Z}_k\}$}
\end{center}
\end{figure}

Note that $Z^R_n$ always belongs to $\{\vartheta\} \cup 0X$. Now we change time to skip the visits of $Z^R_n$ to $\vartheta$ and then look at the {\it jump chain} (see \cite{[N]}). The resulting process has the form $\tilde{Z}_k = Z^R_{T_k}$ where $\{T_k\}$ are suitable stopping times. It can be verified that $\{\tilde{Z}_k\}$ is the simple random walk on $0X$ starting at $0$.

\medskip

Since $\{\tilde{Z}_k\}$ is a simple random walk on $0X$, it converges to some point $\tilde{Z}_{\infty}$ in $[0, \frac{1}{2}]$. Moreover, $\tilde{Z}_k$ converges to a point $x \in [0, \frac{1}{2}]$ if and only if the original walk $Z_n$ converges to either $x$ or $1 - x = Rx$. It follows that
\[
{\mathbb{P}}_{\vartheta} \{\tilde{Z}_{\infty} \in B\} = {\mathbb{P}}_{\vartheta} \{Z_{\infty} \in B \cup RB\}, \ \ \ B \subset [0, \frac{1}{2}].
\]

\begin{figure}[h]
\begin{center}
\includegraphics[width=5.5cm,height=7cm]{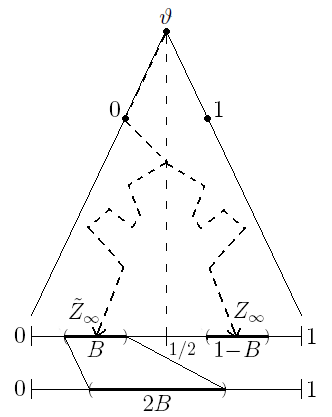}
\caption{Graphical illustration of the identities}
\end{center}
\end{figure}

On the other hand, $[0, \frac{1}{2}]$ can be identified naturally with $[0, 1]$, and we get ${\mathbb{P}}_{\vartheta}\{\tilde{Z}_{\infty} \in B\} = {\mathbb{P}}_{\vartheta}\{Z_{\infty} \in 2B\}= \nu_{\vartheta}(2B)$. Hence we obtain the following \textit{self-similar identity}:
\begin{eqnarray} \label{eqn 3.2}
\nu_{\vartheta}(2B) = \nu_{\vartheta}(B \cup RB), \ \ \ B \subset [0, \frac{1}{2}]  \  \text{ Borel}.
\end{eqnarray}
Now (\ref{eqn 3.1}) and (\ref{eqn 3.2}) imply that $\nu_{\vartheta}$ is the Lebesgue measure on $[0, 1]$. One way to prove this is to show inductively that $\nu_{\vartheta}(I)$ equals the Lebesgue measure of $I$ where $I$ is any dyadic interval.

\medskip

We will now generalize this idea to the Sierpinski gasket of abitrary dimension. We begin by analyzing the symmetries of the $d$-dimensional Sierpinski gasket $K^d$ in relation to the augmented rooted tree $(X, E)$.
\end{section}

\bigskip
\begin{section}{Symmetries and group invariance identity}
Let ${\mathcal{A}}_d$ be the symmetry group of the $d$-dimensional Sierpinski gasket $K = K^d \subset {\mathbb{R}}^d$, i.e., ${\mathcal{A}}_d$ consists of those isometries on ${\mathbb{R}}^d$ that fix $K$. The next lemma is standard.

\begin{Lem} \label{thm 4.1}
${\mathcal{A}}_d$ is isomorphic to the symmetric group ${\mathcal{S}}_{d+1}$ of $d+1$ symbols. For each $g \in {\mathcal{A}}_d$, define $g_K: \{0, ..., d\} \rightarrow \{0, ..., d\}$ by
\[
g_Ki = j \ \ \text{ if } gp_i = p_j.
\]
Then $g_K \in S_{d+1}$, and $g \mapsto g_K$ is an isomorphism. We will identify $g$ and $g_K$. Moreover, the transformation $R_{ij} \in {\mathcal{A}}_d$ corresponding to the transposition $(i, j)$ in ${\mathcal{S}}_{d+1}$ is a reflection in ${\mathbb{R}}^d$. By convention, we set $R_{ii} = id$.
\end{Lem}

\begin{Example} We illustrate the reflection $R_{01}$ on $K^d$ where $d \leq 3$.
\begin{figure}[h]
\begin{center}
\includegraphics[width=11.5cm,height=4.5cm]{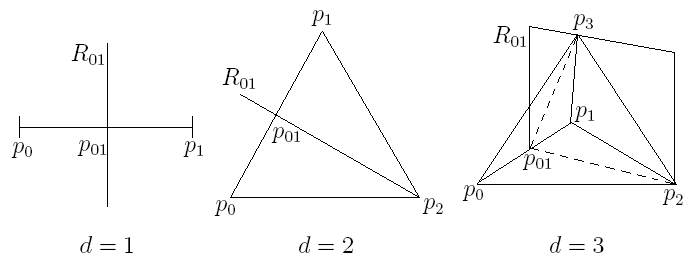}
\caption{The reflection $R_{01}$}
\end{center}
\end{figure}
\end{Example}

We let $\mathrm{Aut}(X)$ be the {\it automorphism group} of $(X, E)$, i.e.
\[
\mathrm{Aut}(X) = \{\varphi: X \rightarrow X | \varphi \text{ is bijective and } {\bf x} \sim {\bf y} \Leftrightarrow \varphi({\bf x}) \sim \varphi({\bf y})\}.
\]
Each $g \in {\mathcal{A}}_d$ also induces an element of $\mathrm{Aut}(X)$ via its action on the cells $K_{{\bf x}}$.

\begin{Lem} \label{thm 4.2}
For $g \in {\mathcal{A}}_d$, define $g_X: X \rightarrow X$ by
\[
g_X{\bf x} = {\bf y}, \ \ \text{if } g(K_{{\bf x}}) = K_{{\bf y}}.
\]
Then $g_X \in \mathrm{Aut}(X)$, and $g \mapsto g_X$ is an injective homomorphism. We will identify $g$ and $g_X$. Moreover, $P$ is invariant under $g$: for all $g \in {\mathcal{A}}_d$
 and ${\bf x}, {\bf y} \in X$,
\[
P({\bf x}, {\bf y}) = P(g{\bf x}, g{\bf y}).
\]
\end{Lem}
\begin{proof}
The first part is standard. To see the second part, note that if ${\bf x} \sim {\bf y}$, then $g{\bf x} \sim g{\bf y}$ and
\[
P({\bf x}, {\bf y}) = \frac{1}{{\rm deg}({\bf x})} = \frac{1}{{\rm deg}(g {\bf x})} = P(g{\bf x}, g{\bf y}).
\]
And if ${\bf x} \nsim {\bf y}$ then both sides are $0$.
\end{proof}

The next lemma says that $g \in {\mathcal{A}}_d$ and $g \in \mathrm{Aut}(X)$ commute in the limit.

\begin{Lem} \label{thm 4.3}
For any Borel set $B$ in $K$ and $g \in {\mathcal{A}}_d$,
\[
\{Z_{\infty} \in B\} = \{g Z_{\infty} \in gB\} = \{\lim_{n \rightarrow \infty} (gZ_n) \in gB\}.
\]
\end{Lem}
\begin{proof}
Let $g \in {\mathcal{A}}_d$ be given. By Lemma \ref{thm 4.2}, $\iota(g{\bf x}) \in K_{g{\bf x}} = g(K_{{\bf x}})$ for all ${\bf x} \in X$. It follows that
\[
|\iota(gZ_n) - g\iota(Z_n)| \leq \mathrm{diam}(K_{Z_n}) \rightarrow 0
\]
as $|Z_n| \rightarrow \infty$. Suppose $Z_{\infty} = y$. Then $\iota(Z_n) \rightarrow y$ and $|Z_n| \rightarrow \infty$ imply that
\[
|\iota(gZ_n) - gy| \leq |\iota(gZ_n) - g\iota(Z_n)| + |g\iota(Z_n) - gy| \rightarrow 0
\]
as $n \rightarrow \infty$. It follows that $gZ_n \rightarrow gy$. Similarly, $\lim_n gZ_n = gy$ implies that $Z_{\infty} = y$.
\end{proof}

\begin{theorem} \label{thm 4.4} {\bf (Group invariance identities)}
The hitting distribution $\nu_{\vartheta}(\cdot) = {\mathbb{P}}_{\vartheta}\{Z_{\infty} \in \cdot\}$ is invariant under the action of ${\mathcal{A}}_d$, i.e., for any Borel set $B$ in $K$,
\[
\nu_{\vartheta}(B) = \nu_{\vartheta}(gB), \ \ \ g \in {\mathcal{A}}_d.
\]
\end{theorem}
\begin{proof}
Fix $g \in {\mathcal{A}}_d$. By Lemma \ref{thm 4.2}, the processes $\{gZ_n\}$ and $\{Z_n\}$ have the same finite dimensional distributions under ${\mathbb{P}}_{\vartheta}$. It follows by a monotone class argument that $\lim_{n \rightarrow \infty} gZ_n$ and $Z_{\infty}$ have the same distribution under ${\mathbb{P}}_{\vartheta}$ as well.

\medskip

Let $B$ be a Borel set in $K$. Then
\begin{eqnarray}
\nu_{\vartheta}(gB) &=& {\mathbb{P}}_{\vartheta}\{ Z_{\infty} \in gB\} \nonumber \\
                    &=& {\mathbb{P}}_{\vartheta}\{g^{-1}Z_{\infty} \in B\} \nonumber \\
                    &=& {\mathbb{P}}_{\vartheta}\{\lim_{n \rightarrow \infty} (g^{-1}Z_n) \in B\} \label{eqn 4.5}\\
                    &=& {\mathbb{P}}_{\vartheta}\{Z_{\infty} \in B\} \label{eqn 4.6}\\
                    &=& \nu_{\vartheta}(B) \nonumber.
\end{eqnarray}

In the above, (\ref{eqn 4.5}) follows from Lemma \ref{thm 4.3}, and (\ref{eqn 4.6}) follows from the fact that $\lim_{n \rightarrow \infty} gZ_n$ and $Z_{\infty}$ have the same distribution.
\end{proof}

\end{section}

\bigskip
\begin{section}{Reflection principle}
In this section we construct the process $\{\tilde{Z}_k\}$ by repeated reflections and show that it is a simple random walk on $0X$. We use $\Omega$ to denote the sample space consisting of all sample paths $\omega$ such that $Z_{n+1}(\omega) \sim Z_n(\omega)$ for all $n$ (note that the simple random walk is of the nearest-neighbor type). For simplicity, we restrict the starting point to be $\vartheta$, and we will work under ${\mathbb{P}}_{\vartheta}$. With slight modification the starting point can be arbitrary.

\medskip

First we introduce some notations. For ${\bf x} \in X \setminus \{\vartheta\}$, we define the {\it parity} $[{\bf x}]$ of ${\bf x}$ as the first symbol of ${\bf x}$. That is, if ${\bf x} = i_1i_2...i_n$, then $[{\bf x}] = i_1$. By convention, we set $[\vartheta] = 0$. For ${\bf x} \in X$, we define a set ${\mathcal{N}}_{{\bf x}}$ by
\[
{\mathcal{N}}_{{\bf x}} = \left\{
              \begin{array}{ll}
                \{j: j \neq i\}, & \hbox{${\bf x} = i$,}\\
                \{ji^{m-1}\}, & \hbox{${\bf x} = ij^{m-1}, i \neq j, \ m \geq 2$,} \\
                \emptyset, & \hbox{otherwise.}
              \end{array}
            \right.
\]
In words, ${\bf y} \in {\mathcal{N}}_{{\bf x}}$ if the parity changes when the walk jumps from ${\bf x}$ to ${\bf y}$ horizontally (${\mathcal{N}}$ stands for {\it neighbor}).

\medskip

For technical convenience, we will first change the time of $\{Z_n\}$. Let $\{{\mathcal{F}}_n\}$ be the standard filtration of $\{Z_n\}$. We define a strictly increasing sequence of $\{{\mathcal{F}}_n\}$-stopping times $\{T_k\}_{k \geq 0}$ by
\begin{eqnarray*}
T_0     &=& \inf\{n \geq 0: Z_n \neq \vartheta\},\\
T_{k+1} &=& \inf\{n > T_k: Z_n \notin \{Z_{T_k}\} \cup \{\vartheta\} \cup {\mathcal{N}}_{Z_{T_k}} \}, \ \ \ k \geq 0.
\end{eqnarray*}
Note that ${\mathbb{P}}_{\vartheta}\{T_0 = 1\} = 1$. By the strong Markov property, the process $\{Y_k\}_{k \geq 0}$ defined by
\[
Y_k = Z_{T_k}, \ \ \ k \geq 0,
\]
is a Markov chain on $X \setminus \{\vartheta\}$ with respect to $\{{\mathcal{G}}_k\}_{k \geq 0}$, where ${\mathcal{G}}_k = \sigma(Z_{T_0}, ..., Z_{T_k})$. From the construction, we see that if $[Y_{k+1}] = j \neq i = [Y_k]$, then $R_{ij}Y_{k+1} \neq Y_k$. (This ensures that the reflected process does not stay. See Example \ref{eg 4.2}.)

\medskip

The transition function of $\{Y_k\}$ is given in the next lemma, which is the main ingredient of the proof of Proposition \ref{thm 5.2}.

\begin{Lem} \label{thm 5.1}
Let ${\bf x} \in X \setminus \{ \vartheta \}$.
\begin{enumerate}
\item[(i)] Suppose ${\bf x}$ is not of the form $ij^{m-1}$, where $i \neq j$ and $m \geq 1$, and ${\bf y} \sim {\bf x}$. Then
\[
{\mathbb{P}}_{\vartheta}\{Y_{k+1} = {\bf y} | Y_k = {\bf x}\} = \frac{1}{\mathrm{deg}({\bf x})}.
\]
\item[(ii)] Suppose ${\bf x} = i$ for some $i$. Let $a \in \{0, ..., d\}$. Then
\[
{\mathbb{P}}_{\vartheta}\{Y_{k+1} \in \{R_{ij}({\bf x}a): j = 0, ..., d\} |Y_k = {\bf x}\}  = \frac{1}{d + 1}.
\]
\item[(iii)] Suppose ${\bf x} = ij^{m-1}$, $i \neq j$ and $m \geq 2$. Let ${\bf y} \in iX$ with ${\bf y} \sim {\bf x}$. Then
\[
{\mathbb{P}}_{\vartheta}\{Y_{k+1} \in \{{\bf y}, R_{ij}{\bf y}\} | Y_k = {\bf x}\} = \frac{1}{\mathrm{deg}({\bf x}) - 1}.
\]
\end{enumerate}
\end{Lem}

\begin{figure}[h]
\begin{center}
\includegraphics[width=11.5cm,height=4.5cm]{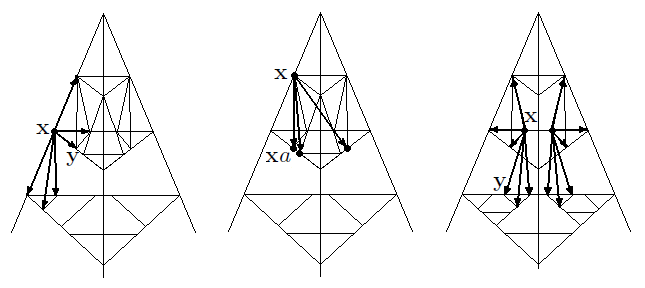}
\caption{Illustration of Lemma \ref{thm 5.1} (i) (left), (ii) (middle) and (iii) (right)}
\end{center}
\end{figure}

\begin{proof}
To prove (i), we assume ${\bf x}$ is not of the form $ij^{m-1}$. Then on the event $\{Y_k = {\bf x}\}$, we have $T_{k+1} = T_k + 1$, and so $Y_{k+1} = Z_{T_k + 1}$. It follows from the strong Markov property that
\[
{\mathbb{P}}_{\vartheta}\{Y_{k+1} = {\bf y} | Y_k = {\bf x}\} = {\mathbb{P}}_{\vartheta}\{Z_{{T_k}+1} = {\bf y} | Z_{T_k} = {\bf x}\} = \frac{1}{\mathrm{deg}({\bf x})}.
\]
For (ii), let $B = \{R_{ij}({\bf x}a): j = 0, ..., d\}$, $\tau_2 = \inf \{n \geq 0: |Z_n| = 2\}$ and $q_{{\bf z}} = {\mathbb{P}}_{{\bf z}}\{Z_{\tau_2} \in B\}$. Then ${\mathbb{P}}_{\vartheta}\{Y_{k+1} \in B | Y_k = {\bf x}\} = q_{{\bf x}}$. By standard first-step calculations, we have
\begin{eqnarray*}
q_{\vartheta} &=& \sum_{j = 0}^d \frac{1}{d+1}q_j,\\
q_j      &=& \frac{1}{2d+2}1 + \frac{1}{2d+2}q_{\vartheta} + \sum_{l = 0, l \neq j}^d \frac{1}{2d+2}q_l, \ \ \  j = 0, ..., d.
\end{eqnarray*}
Note that ${\mathrm{deg}}(j) = 1 + d + (d+1) = 2d+2$. Solving the equations, we get $q_j = q_{\vartheta} = \frac{1}{d+1}$ for all $j$.

\noindent
The proof of (iii) is similar to that of (ii).
\end{proof}

To define the reflected process $\{\tilde{Z}_k\}$, we need to keep track of the parity changes of $\{Y_k\}$. We define a strictly increasing sequence $\{S_p\}_{p \geq 0}$ of $\{{\mathcal{G}}_m\}$-stopping times by $S_0 = 0$,
\[
S_{p + 1} = \inf\{m > S_p: [Y_m] \neq [Y_{S_p}]\}, \ \ \ p \geq 0.
\]
Also define a random sequence $\{G_p\}$ in ${\mathcal{A}}_d$ by $G_0 = R_{0, [Y_0]}$,
\[
G_{p+1} = G_p \circ R_{[Y_{S_p}], [Y_{S_{p+1}}]}, \ \ \ p \geq 0.
\]
Hence
\[
G_p = R_{0, [Y_0]} \circ R_{[Y_0], [Y_{S_1}]} \circ \cdots \circ R_{[Y_{S_{p-1}}], [Y_{S_p}]}
\]
is a random product of reflections induced by $\{Y_k\}$. We leave $G_p$ undefined if $S_p = \infty$. Finally, for $k \geq 0$, we define
\[
\tilde{Z}_k = G_pY_k, \ \ \ \text{ if } S_p \leq k < S_{p+1}.
\]
By induction, one can check that $\tilde{Z}_k \in 0X$ for all $k \geq 0$. For each $k$, let $L_k$ be the unique random integer such that $L_k(\omega) = p$ if $S_p(\omega) \leq k < S_{p+1}(\omega)$. Then $\tilde{Z}_k = G_{L_k}Y_k$ for all $k$. See Figure 7 for an illustration.

\begin{figure}[h]
\begin{center}
\includegraphics[width=9cm,height=4cm]{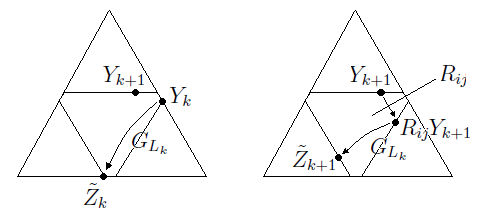}
\caption{Suppose that $[Y_k] = i$ and $[Y_{k+1}] = j$, where $i \neq j$. Then $\tilde{Z}_{k+1}$ is defined as $\tilde{Z}_{k+1} = (G_{L_k} \circ R_{ij}) Y_k$.}
\end{center}
\end{figure}

\begin{Example} \label{eg 4.2}
Consider the case $d = 1$. We compute $Y_k$, $G_{L_k}$ and $\tilde{Z}_k$ for the following sample path of $\{Z_n\}_{n = 0}^8$:
\begin{center}
\begin{tabular}{|c|c|c|c|c|c|c|c|c|c|}
  \hline
 $n$           & $0$ & $1$ & $2$ & $3$ & $4$ & $5$ & $6$ & $7$ & $8$ \\
  \hline
 $Z_n$         & $\vartheta$ & $0$ & $\vartheta$ & $1$ & $10$ & $100$ & $001$ & $01$ & $00$ \\
  \hline
 $k$           &   & $0$ &   &   & $1$ & $2$ &    & $3$ & $4$ \\
  \hline
 $Y_k$         &   & $0$ &   &   & $10$& $100$ &   & $01$  & $00$\\
  \hline
 $L_k$         &   & $0$  &   &  & $1$  & $1$  &  & $2$  & $2$ \\
  \hline
 $G_{L_k}$     &   & $id$  &   &   & $R_{01}$  & $R_{01}$  &  & $id$  & $id$ \\
  \hline
 $\tilde{Z}_k$ &   & $0$  &   &   & $01$  & $011$ &   & $01$  & $00$ \\
  \hline
\end{tabular}
\end{center}

\begin{figure}[h]
\begin{center}
\includegraphics[width=11.5cm,height=5cm]{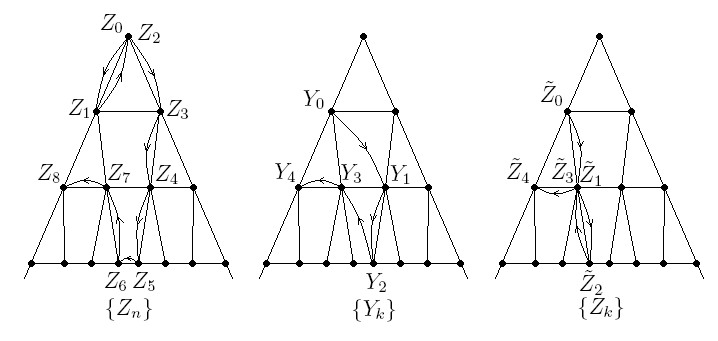}
\caption{Illustration of the processes $\{Z_n\}$, $\{Y_k\}$ and $\{\tilde{Z}_k\}$}
\end{center}
\end{figure}

\end{Example}

The main result of this section is the following.

\begin{Prop} \label{thm 5.2}
Under ${\mathbb{P}}_{\vartheta}$, $\{\tilde{Z}_k\}_{k \geq 0}$ is a simple random walk on $0X$ with ${\mathbb{P}}_{\vartheta}\{\tilde{Z}_0 = 0\} = 1$.
\end{Prop}
\begin{proof} We first check that ${\mathbb{P}}_{\vartheta}\{\tilde{Z}_0 = 0\} = 1$. Under ${\mathbb{P}}_{\vartheta}$, we have almost surely that $T_0 = 1$. Hence $Y_0 = Z_1$ and $G_{L_0} = G_0 = R_{0, [Y_0]}$. It follows that $\tilde{Z}_0 = R_{0, [Y_0]}Y_0 = 0$ almost surely.

\medskip

\noindent
Next we prove by induction on $N$ that
\[
{\mathbb{P}}_{\vartheta}\{\tilde{Z}_k = {\bf z}_k, 0 \leq k \leq N\} = \prod_{k = 0}^{N-1} \frac{1}{\mathrm{deg}_0({\bf z}_k)}
\]
for any path $\{{\bf z}_k\}_{k = 0}^N$ in $0X$ such that ${\bf z}_0 = 0$. Here $\mathrm{deg}_0$ denotes the degree in the subgraph $0X$, and, by convention, the product is $1$ when $N = 0$. This establishes that $\{\tilde{Z}_k\}$ is a simple random walk on $0X$.

\medskip

\noindent
The case $N = 0$ is trivial. Assume the claim for paths of length $N$ and consider the probability
\[
{\mathbb{P}}_{\vartheta}\{\tilde{Z}_k = {\bf z}_k, 0 \leq k \leq N + 1\}.
\]
The idea is to condition on the value of $G_{L_N}$. By iterated expectation,
\begin{eqnarray}
& & {\mathbb{P}}_{\vartheta}\{\tilde{Z}_k = {\bf z}_k, 0 \leq k \leq N + 1\} \nonumber \\
&=& {\mathbb{E}}_{\vartheta}\left[{\mathbb{E}}_{\vartheta} \left( \prod_{k = 0}^{N+1} {\bf 1}\{\tilde{Z}_k = {\bf z}_k\} | {\mathcal{G}}_N \right) \right] \nonumber \\
&=& \sum_{g \in {\mathcal{A}}_d} {\mathbb{E}}_{\vartheta} \left[\prod_{k = 0}^N {\bf 1}\{\tilde{Z}_k = {\bf z}_k\} {\bf 1}\{G_{L_N} = g\} {\mathbb{E}}_{\vartheta} \left( {\bf 1}\{G_{L_{N+1}}Y_{N+1} = {\bf z}_{N + 1}\} | {\mathcal{G}}_N \right) \right].  \label{eqn 5.1}
\end{eqnarray}
In the last equality we use the fact that $\{G_{L_k}\}$ and $\{\tilde{Z}_k\}$ are adapted to $\{{\mathcal{G}}_k\}$.

\medskip
\noindent
Now we distinguish three cases and use Lemma \ref{thm 5.1}.\\
(i) ${\bf z}_N$ is not of the form $0j^{m-1}$ where $j \neq 0$ and $m \geq 1$. Then on the event $\{\tilde{Z}_N = {\bf z}_N, G_{L_N} = g\}$ we have $Y_N = g^{-1}\tilde{Z}_N$ and $G_{L_{N+1}} = G_{L_N} = g$. It follows that on the event $\{\tilde{Z}_N = {\bf z}_N, G_{L_N} = g\}$ we have
\begin{eqnarray}
& & {\mathbb{E}}_{\vartheta}[{\bf 1}\{G_{L_{N+1}}Y_{N+1} = {\bf z}_{N + 1}\} | {\mathcal{G}}_K] \nonumber \\
&=& {\mathbb{P}}_{\vartheta}\{gY_{N+1} = {\bf z}_{N + 1} | {\mathcal{G}}_N\} \nonumber \\
&=& {\mathbb{P}}_{\vartheta}\{gY_{N+1} = {\bf z}_{N + 1} | Y_N\} \label{eqn 5.2}\\
&=& {\mathbb{P}}_{\vartheta}\{Y_{N+1} = g^{-1}{\bf z}_{N + 1} | Y_N = g^{-1}{\bf z}_N\} \nonumber \\
&=& \frac{1}{\mathrm{deg}(g^{-1}{\bf z}_N)} \label{eqn 5.3} \\
&=& \frac{1}{\mathrm{deg}({\bf z}_N)}. \nonumber
\end{eqnarray}
In the above, equality in (\ref{eqn 5.2}) is the simple Markov property of $\{Y_k\}$, and (\ref{eqn 5.3}) follows from Lemma \ref{thm 5.1}(i). Note that $\mathrm{deg}(g^{-1}{\bf z}_N) = \mathrm{deg}({\bf z}_N) = \mathrm{deg}_0({\bf z}_N)$. Putting this into (\ref{eqn 5.1}) and continuing the calculation, we get from the induction hypothesis that
\begin{eqnarray*}
& & {\mathbb{P}}_{\vartheta}\{\tilde{Z}_k = {\bf z}_k, 1 \leq k \leq N + 1\} \\
&=& \sum_{g \in {\mathcal{A}}_d} {\mathbb{E}}\left[\prod_{k = 0}^N {\bf 1}\{\tilde{Z}_k = {\bf z}_k\} {\bf 1}\{G_{L_N} = g\} \frac{1}{\mathrm{deg}_0({\bf z}_N)}\right]\\
&=& \frac{1}{\mathrm{deg}_0({\bf z}_N)} {\mathbb{E}}\left[\prod_{k = 0}^N 1\{\tilde{Z}_k = {\bf z}_k\}\right]\\
&=& \prod_{k = 0}^N \frac{1}{\mathrm{deg}_0({\bf z}_k)}.
\end{eqnarray*}
The remaining cases are similar.\\
(ii) ${\bf z}_N = 0$. On the event $\{\tilde{Z}_N = {\bf z}_N, G_{L_N} = g\}$, we have
\begin{eqnarray*}
&& {\mathbb{E}}_{\vartheta}[{\bf 1}\{\tilde{Z}_{N+1} = {\bf z}_{N+1}\}|{\mathcal{G}}_N] \\
&=& {\mathbb{P}}_{\vartheta}\{Y_{N+1} \in \bigcup_{j = 0}^d \{ (g \circ R_{[g^{-1}{\bf z}_N],j})^{-1} {\bf z}_{N+1}\} | Y_N = g^{-1}{\bf z}_N\},
\end{eqnarray*}
and by Lemma \ref{thm 5.1}(ii) this equals $\frac{1}{d+1} = \frac{1}{{\mathrm{deg}}_0(0)}$.\\
(iii) ${\bf z}_N = 0j^{m - 1}$ where $j \neq 0$ and $m \geq 2$. If $G_{L_N} = g$, then $g^{-1}{\bf z}_N$ has a unique horizontal neighbor whose parity (say $i$) is different from that of $g^{-1}{\bf z}_N$. Then, on the event $\{\tilde{Z}_N = {\bf z}_N, G_{L_N} = g\}$, we have
\begin{eqnarray*}
&& {\mathbb{E}}_{\vartheta}[{\bf 1}\{\tilde{Z}_{N+1} = {\bf z}_{N+1}\}|{\mathcal{G}}_N] \\
&=& {\mathbb{P}}_{\vartheta}\{Y_{N+1} \in \{g^{-1}{\bf z}_{N+1}, (g \circ R_{i, [g^{-1}{\bf z}_N]})^{-1}{\bf z}_{N+1} | Y_N = g^{-1}{\bf z}_N\},
\end{eqnarray*}
which equals $\frac{1}{{\mathrm{deg}}({\bf z}_N) - 1} = \frac{1}{{\mathrm{deg}_0}({\bf z}_N)}$ by Lemma \ref{thm 5.1}(iii).
\end{proof}

\end{section}

\bigskip
\begin{section}{Self-similar identity and hitting distribution}
First we relate the hitting distribution of $\{\tilde{Z}_k\}$ with that of $\{Z_n\}$. Since $0X$ is isomorphic to $X$, we may regard $0X$ as the augmented rooted tree of the IFS $\{F_{0i}\}_{i = 0}^d$ with self-similar set $K_0$. It follows that $\{\tilde{Z}_k\}$, which is a simple random walk on $0X$, converges almost surely to a point on $K_0$. Thus we get:

\begin{Lem} \label{thm 6.1}
Under ${\mathbb{P}}_{\vartheta}$, $\tilde{Z}_k$ converges almost surely to a point $\tilde{Z}_{\infty}$ on $K_0 \subset K$. For any Borel set $B$ in $K_0$, we have
\[
{\mathbb{P}}_{\vartheta}\{\tilde{Z}_{\infty} \in B\} = {\mathbb{P}}_{\vartheta}\{Z_{\infty} \in F_0^{-1}(B)\} = \nu_{\vartheta}(F_0^{-1}(B)).
\]
\end{Lem}

\medskip

Next we consider the relation between $Z_{\infty}$ and $\tilde{Z}_{\infty}$. Recall that $p_{ij}$ is the midpoint of the vertices $p_i$ and $p_j$.

\begin{Lem} \label{thm 6.2}
There exists $G_{\infty} \in {\mathcal{A}}_d$, depending on $\omega$, such that $\tilde{Z}_{\infty} = G_{\infty} Z_{\infty}$.
\end{Lem}
\begin{proof}
Fix $\omega$ such that $Z_{\infty}(\omega)$ exists and we will suppress $\omega$ in what follows. Since $Y_k$ is a subsequence of $Z_n$, we have
\[
Z_{\infty} = Y_{\infty} := \lim_{k \rightarrow \infty} Y_k.
\]
First suppose that $\tilde{Z}_{\infty} \neq p_{0i}$ for all $i \neq 0$. We claim that there exists $N$ such that $\rho(Y_k) = \rho(Y_N)$ for all $k \geq N$. That is, the parity stays constant for $k$ large enough. To see this, note that the parity changes only when $\tilde{Z}_k = 0i^l$ for some $i$ and $l$. Now since $\tilde{Z}_{\infty} \neq p_{0i}$, there is some $\epsilon > 0$ and $L \in {\mathbb{N}}$ such that when $k$ is large, ${\mathrm{dist}}(\iota(\tilde{Z}_k), K_{0i^L}) \geq \epsilon$ for all $i$. It follows that $\tilde{Z}_k \neq 0i^l$, where $l \geq L$, when $k$ is large. Hence $G_{L_k} = G_{L_N} =: G_{\infty}$ for all $k \geq N$ and
\[
\tilde{Z}_k = G_{\infty}Y_k, \ \ \ k \geq N.
\]
Letting $k \rightarrow \infty$, we get $\tilde{Z}_{\infty} = G_{\infty}Y_{\infty} = G_{\infty}Z_{\infty}$.

\medskip

Next suppose that $\tilde{Z}_{\infty} = p_{0i}$ for some $i$. Since every $g \in {\mathcal{A}}_d$ must map $p_{0i}$ to some other midpoint, by convergence of $\{Y_k\}$ we see that for large $k$, $[Y_k]$ takes at most two values, say $j$ and $l$. Hence, there exists $G \in {\mathcal{A}}_d$ such that for each sufficiently large $k$, $G_{L_k}$ is either $G$ or $G \circ R_{jl}$. We may take $G_{\infty}$ to be any one which appears infinitely often.
\end{proof}

Let ${\mathcal{A}}'_d = \{g \in {\mathcal{A}}_d: gK_0 = K_0\}$ be the subgroup of ${\mathcal{A}}_d$ that fixes the cell $K_0$. It corresponds to the subgroup of ${\mathcal{S}}_{d+1}$ that fixes the symbol $0$ and hence is isomorphic to ${\mathcal{S}}_d$. The next lemma, which is purely geometric, is straightforward to prove.

\begin{Lem} \label{thm 6.3}
Suppose $B \subset K_0$ is Borel and is invariant under ${\mathcal{A}}'_d$, i.e., $gB = B$ for all $g \in {\mathcal{A}}'_d$. Then $\bigcup_{i = 0}^d R_{0i}B$ is invariant under ${\mathcal{A}}_d$, i.e.,
\[
g(\bigcup_{i = 0}^d R_{0i}B) = \bigcup_{i = 0}^d R_{0i}B, \ \ \ g \in {\mathcal{A}}_d.
\]
Moreover, $\left( \bigcup_{i = 0}^d R_{0i} B\right) \cap K_0 = B$.
\end{Lem}

The reason of using ${\mathcal{A}}'_d$-invariant sets is that for fixed $Z_{\infty}$, $\tilde{Z}_{\infty} = G_{\infty}Z_{\infty}$ may take one of several values depending on the sequence of reflections made.

\begin{Lem} \label{thm 6.4}
Suppose $B \subset K_0$ is Borel and is ${\mathcal{A}}'_d$-invariant. Then
\[
\{\tilde{Z}_{\infty} \in B\} = \{Z_{\infty} \in \bigcup_{i = 0}^d R_{0i} B\}.
\]
\end{Lem}
\begin{proof}
By Lemma \ref{thm 6.2}, $\tilde{Z}_{\infty} = G_{\infty} Z_{\infty}$ for some (random) $G_{\infty} \in {\mathcal{A}}_d$. If $\tilde{Z}_{\infty} \in B$, then
\[
Z_{\infty} = G_{\infty}^{-1} \tilde{Z}_{\infty} \in G_{\infty}^{-1} B \subset \bigcup_{i = 0}^d R_{0i}B
\]
by Lemma \ref{thm 6.3}. On the other hand, suppose $Z_{\infty} \in \bigcup_{i = 0}^d R_{0i}B$ for some $i$ and let $G_{\infty}$ be as above. By Lemma \ref{thm 6.3} again, we have
\[
\tilde{Z}_{\infty} = G_{\infty} Z_{\infty} \in \left(G_{\infty} \bigcup_{i = 0}^d R_{0i}B \right) \cap K_0 = B.
\]
\end{proof}

By Lemmas \ref{thm 6.1} and \ref{thm 6.4}, we immediately obtain the following crucial result.

\begin{theorem} \label{thm 6.5} {\bf (Self-similar identity)}
Suppose $B \subset K_0$ is Borel and is ${\mathcal{A}}'_d$-invariant, i.e., $gB = B$ for all $g \in {\mathcal{A}}'_d$. Then
\[
\nu_{\vartheta}(S_0^{-1}B) = \nu_{\vartheta}(\bigcup_{i = 0}^d R_{0i}B).
\]
\end{theorem}

\begin{figure}[h]
\begin{center}
\includegraphics[width=7.5cm,height=3.5cm]{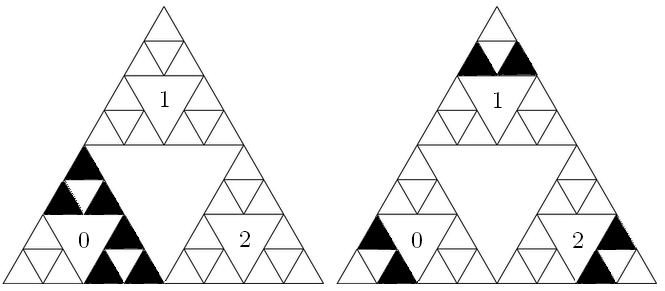}
\caption{The two sets have the same probability under $\nu_{\vartheta}$.}
\end{center}
\end{figure}

\medskip

We are now ready to prove the main theorem.

\medskip

\noindent
{\it Proof of Theorem \ref{thm 1.1}.} Let $\mu$ be the normalized $\alpha$-dimensional Hausdorff measure on $K$, where $\alpha = \dim_H K$. We will show that $\mu$ is the unique Borel probability measure on $K$ that satisfies the identities in Theorems \ref{thm 4.4} and \ref{thm 6.5}. Since $\nu_{\vartheta}$ has been shown to satisfy these identifies, this implies that $\nu_{\vartheta} = \mu$.

\medskip

Let $\lambda$ be any Borel probability measure on $K$ satisfying the identities. We will complete the proof assuming the claim that $\lambda$ has no atoms on the dyadic points, i.e., $\lambda(p_{{\bf x}}) = 0$ for all ${\bf x} \in X \setminus \{\vartheta\}$. This allows us to use additivity for sets that intersect only at dyadic points.

\medskip

It suffices to show that $\lambda(K_{{\bf x}}) = \frac{1}{(d+1)^{|{\bf x}|}}$ for all ${\bf x} \in X$. We proceed by induction on $|{\bf x}|$, the length of ${\bf x}$. By definition of $\lambda$, we have $\lambda(K) = 1$. For the first level, we have
\[
1 = \lambda(K) = \lambda(\bigcup_{i = 0}^d K_i) = \lambda(K_0) +  \sum_{i = 1}^d \lambda(R_{0i} K_0) = (d+1) \lambda(K_0),
\]
where in the last equality we used group invariance and the fact that $R_{0i} K_0 = K_i$. Hence $\lambda(K_i) = 1/(d+1)$ for all $i$.

\medskip

Now suppose that $\lambda(K_{{\bf x}}) = \frac{1}{(d+1)^{|{\bf x}|}}$ for all ${\bf x}$ with $|{\bf x}| = m$. Applying the self-similar identity, we have
\[
\frac{1}{(d+1)^m} = \lambda(K_{0^m}) = \lambda(\bigcup_{i = 0}^d K_{i^{m+1}}) = (d+1) \lambda(K_{0^{m+1}}).
\]
Hence $\lambda(K_{0^{m+1}}) = \frac{1}{(d+1)^m}$. Next, fix any $\bf x$ with $|{\bf x}| = m$ and ${\bf x} \neq 0^m$. Consider the set
\[
\bigcup_{g \in {\mathcal{A}}'_d} gK_{0{\bf x}} \subset K_0.
\]
Observe that the sets in the union intersect only at dyadic points and is invariant under ${\mathcal{A}}'_d$. By the self-similar identity, we have
\[
\lambda(\bigcup_{g \in {\mathcal{A}}'_d} gK_{\bf x}) = \lambda(S_0^{-1} \bigcup_{g \in {\mathcal{A}}'_d} gK_{0{\bf x}}) = \lambda(\bigcup_{i = 0}^d \bigcup_{g \in {\mathcal{A}}'_d} gK_{0{\bf x}}).
\]
Using group invariance and the fact that $\#({\mathcal{A}}'_d) = d!$, we get
\[
\frac{d!}{(d+1)^m} = (d+1) d! \lambda(K_{0{\bf x}}).
\]
It follows that $\lambda(K_{0{\bf x}}) = \frac{1}{(d+1)^{m+1}}$ for all $|{\bf x}| = m$, and applying $R_{0i}$, $i \neq 0$, this implies that $\lambda(K_{{\bf y}}) = \frac{1}{(d+1)^{m+1}}$ for all $|{\bf y}| = m+1$. This completes the induction argument.

\medskip

It remains to verify the claim that $\lambda(p_{{\bf x}}) = 0$ for all ${\bf x}$. The proof is to show by induction on the level of ${\bf x}$ that all $p_{\bf x}$ have equal probability. The method is the same as the above induction and we leave the argument to the reader. And since the set of dyadic points is infinite, this implies that the probability is $0$.

\begin{flushright}
$\Box$
\end{flushright}
\end{section}

\bigskip
\begin{section}{Remarks and open questions}
There are many more things that can be said in this model. Using the reflection principle and induction on $N$, we can show the following:

\begin{Prop}
For $N \geq 0$, let $\tau_N = \inf\{n \geq 0: |Z_n| = N\}$. Then under ${\mathbb{P}}_{\vartheta}$, the distribution of $Z_{\tau_N}$ is uniform on $\{0, ..., d\}^N$.
\end{Prop}

The idea of the proof is to reformulate the group invariance and self-similar identities in terms of $Z_{\tau_N}$ and $\tilde{Z}_{\tau_N}$. For example, the group invariance identity will take the form
\[
{\mathbb{P}}_{\vartheta}\{Z_{\tau_N} \in B\} = {\mathbb{P}}_{\vartheta}\{Z_{\tau_N} \in gB\}, \ \ \ B \subset \{0, ..., d\}^N.
\]
Since $\tilde{Z}_k$ is the simple random walk on $0X$ starting at $0$, the distribution of $\tilde{Z}_{\tau_N}$ can be expressed by that of $Z_{\tau_{N-1}}$. This allows us to use the induction hypothesis (and for small $N$ the proposition follows by direct calculations). Together with a limiting argument, this gives an alternative approach to Theorem \ref{thm 1.1}.

\medskip

We have only shown that $\nu_{\vartheta} = \mu$. How about $\nu_{{\bf x}}$ for ${\bf x} \neq \vartheta$? As remarked in the beginning of Section 5, the reflection principle can be formulated for any starting point. For other starting points, the same method can be used to show the following:

\begin{Prop}
Let ${\bf x} = 0{\bf y} \in 0X$, where ${\bf y} \in X$. Define $\sigma {\bf x} = {\bf y}$. Then $\nu_{\sigma {\bf x}}(S_0^{-1}B) = \nu_{{\bf x}}(\bigcup_{i = 0}^d R_{0i}B)$ for all $B \subset K_0$ which is invariant under ${\mathcal{A}}'_d$.
\end{Prop}

\medskip
\noindent
{\bf Q1}. Can the above be used to derive estimates of $\nu_{{\bf x}}$?

\medskip

Our method does not require any estimate of the Martin kernel, and this is both an advantage and a disadvantage. Even for the simplest case $d = 1$, we have not been able to derive good estimates of the Martin kernels.

\medskip
\noindent
{\bf Q2}. When $d = 1$, can we prove directly (without using the hyperbolic boundary as in \cite{[WL]}) that the Martin boundary is homeomorphic to $[0, 1]$?
\bigskip

It is natural to consider generalization of the reflection principle to other highly symmetric augmented rooted trees. As we have seen, the crucial objects needed is that each $iX$ is isomorphic to $X$, and there is a `local' reflection $R_{ij} \in \mathrm{Aut}(iX \cup jX)$ (and a corresponding map on $K$) whenever $i \neq j$. Indeed, there are many cases where our method is applicable. An example is the {\it pentagon fractal}.
\begin{figure}[h]
\begin{center}
\includegraphics[width=4.5cm,height=4.5cm]{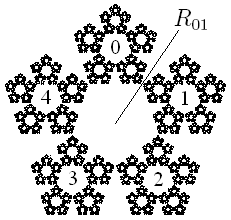}
\caption{The pentagon fractal with a reflection $R_{01}$.}
\end{center}
\end{figure}

Here the symmetry group is the dihedral group $D_5$. Another example is {\it Lindstr{\o}m's snowflake} (see \cite{[Ku]}), where the symmetry group is the dihedral group $D_6$. For both cases the hitting distribution starting from the root $\vartheta$ is uniform. The arguments are the same and we leave the details to the reader. So far we have not been able to generalize this to an axiomatic framework such as the class of {\it nested fractals} (see \cite{[Ku]}).

\medskip
\noindent
{\bf Q3}. Can the results of this paper be generalized to all nested fractals?

\medskip

We may think about the simple random walks on Sierpinski graphs as a special case of the framework in \cite{[WL]}. In that general setting, very little is known about the hitting distributions. As is mentioned in the introduction, these measures serve as the reference measures of the induced Dirichlet forms in \cite{[WL]}, and understanding them is of value to analyze the associated jump processes.

\medskip
\noindent
{\bf Q4}. For the simple (or a strictly reversible, see \cite{[WL]}) random walk on the augmented rooted tree of a self-similar set satisfying OSC, are the hitting distributions absolutely continuous with respect to the Hausdorff (or a self-similar) measure?

\end{section}

\bigskip

\noindent {\bf Acknowledgements}\\
The author would like to thank Professor Ka-Sing Lau for his guidance and support during the preparation of this paper.

\end{document}